\documentclass[12pt]{amsart}
\usepackage{amsfonts,amssymb,amscd,amsmath,amsthm,enumerate,verbatim,calc}

\theoremstyle{plain}
\newtheorem{Theorem}{Theorem}[section]

\newtheorem{Lemma}[Theorem]{Lemma}

\newtheorem{Proposition}[Theorem]{Proposition}
\newtheorem{Question}{Question}

\theoremstyle{definition}
\newtheorem{Definition}[Theorem]{Definition}

\newtheorem{Remark}[Theorem]{Remark}

\theoremstyle{remark}

\newtheorem*{chunk*}{}

\numberwithin{equation}{Theorem}

\newcommand{\mlabel}[1]%
{\mbox{}\marginpar{\raggedleft\hspace{0pt}{\rm\ttfamily#1}}\label{#1}}

\newcommand{\depth}{\operatorname{depth}}

\newcommand{\het}{\operatorname{ht}}

\newcommand{\ext}{\operatorname{Ext}}
\newcommand{\Coker}{\operatorname{Coker}}

\newcommand{\Supp}{\operatorname{Supp}}

\newcommand{\Spec}{\operatorname{Spec}}
\newcommand{\fm}{{\mathfrak m}}
\newcommand{\fn}{{\mathfrak n}}
\newcommand{\fp}{{\mathfrak p}}

\newcommand{\fq}{{\mathfrak q}}

\newcounter{hours}\newcounter{minutes}

\newcommand{\excise}[1]{}

\begin{document}
\title[On the localization theorem for F-pure rings]
{On the localization theorem for F-pure rings}
\author[K.Shimomoto]{Kazuma Shimomoto}
\author[W.Zhang]{Wenliang Zhang}

\address{483-2 Kiyotaki Shijonawate-shi, Osaka 575-0063, Japan}
\email{shimomotokazuma@gmail.com}

\address{Department of Mathematics, University of Michigan, Ann Arbor, MI 48109, USA}
\email{wlzhang@umich.edu}



\maketitle

\begin{abstract}
In this article, we solve Grothendieck's localization problem for a certain class of rings that arises from tight closure theory. It is quite essential to use the relative version of the Frobenius map (Radu-Andr$\mathrm{\grave{e}}$ morphism) to investigate the fibers of a flat ring homomorphism of characteristic $p > 0$. We also give some interesting geometric applications of the localization problem.
\end{abstract}

\bigskip

\section{Introduction}

Let  $R$ be a discrete valuation ring and let $\varphi:R \to S$ be a flat ring homomorphism. Then one of the interesting questions in this situation is to study how the generic fiber of $\varphi$ degenerates to the closed fiber. In general, even if one starts with a smooth generic fiber, the closed fiber can absorb bad singularities. The localization problem, roughly speaking, describes the effect of the closed fiber to the general fibers under a flat local map of local rings.

For a homomorphism $\varphi:R \to S$ of noetherian rings, we denote by $k(\fp)$ the residue class field for $\fp \in \Spec R$. Then the fiber ring of the map $\varphi$ at $\fp$ is defined to be $S \otimes_{R} k(\fp)$. Let $R$ be a local ring with its completion $\widehat{R}$. Then the formal fiber of $R$ at $\fp \in \Spec R$ is the fiber of the completion map $R \to \widehat{R}$ at $\fp$. With a bit ambiguity, the formal fiber rings of $R$ (not necessarily a local ring) at $\fp$ are the set of all fiber rings of the completion map $R_{\fp} \to \widehat{R_{\fp}}$, where $\widehat{R_{\fp}}$ is the $\fp$-adic completion of $R_{\fp}$. Let $\mathcal{P}$ be some ring theoretic property defined for commutative rings. Then Grothendieck~\cite[7.5.4]{Gro} posed the following question:

\begin{Question}[Grothendieck's Localization Problem]
Let $\varphi:R \to S$ be a flat local homomorphism of noetherian rings such that the closed fiber ring of $\varphi$ and all the formal fiber rings of $R$ with respect to all $\fp \in \Spec R$ satisfy $\mathcal{P}$. Then does every fiber ring of $\varphi$ satisfy $\mathcal{P}$?
\end{Question}

The localization problem has been solved in many interesting cases and the most comprehensive work was made by Avramov and Foxby (\cite{AF1},~\cite{AF2} for more results), where they established the localization problem for many classes of rings by introducing various kinds of numerical invariants, known as defects. Especially, their techniques allow them to solve the problem in the case of Cohen-Macaulay, Gorenstein, and complete intersection rings.

The main focus in this paper is on the class of rings that arise from tight closure theory. We give a quick review for tight closure theory in the next section. For details of the story, we refer the reader to the monograph~\cite{Hu}. Let us denote by $\mathcal{P}$ a property of noetherian rings that is related with tight closure theory, such as being $F$-pure, $F$-injective, and so on. Let $R$ be an algebra over a field $K$. Then we say that $R$ has $\textit{geometrically}$ $\mathcal{P}$ if $R \otimes_{K} L$ has $\mathcal{P}$ for every finite field extension $L$ of $K$.

\begin{Question}
\label{ourquestion}
Let $\varphi:R \to S$ be a flat local map of $F$-finite rings. Assume that the closed fiber ring of $\varphi$ has $($geometrically$)$ $\mathcal{P}$. Then does every fiber ring of $\varphi$ have also $($geometrically$)$ $\mathcal{P}$?
\end{Question}
In Question \ref{ourquestion}, no conditions on formal fibers of $R$ are imposed, since every $F$-finite ring is excellent (cf. \cite[2.5]{Ku}) and all formal fibers of an excellent ring are geometrically regular.

In recent years, some strong evidence has been found that shows that tight closure theory tends to behave naturally for $F$-finite rings (or more generally, excellent rings). For this reason, we limit ourselves to $F$-finite rings.

In this paper, we solve Grothendieck's localization problem for ``$\mathcal{P}$=maximal Cohen-Macaulay'' (Proposition \ref{localization-mcm}) and ``$\mathcal{P}$=geometrically $F$-pure" (Theorem \ref{localization}) and establish some geometric consequences of our results on maximal Cohen-Macaulayness (Proposition \ref{generic-mcm}) and geometric $F$-purity (Theorem \ref{generic-f-pure}).

\section{Preliminaries and notations}

Throughout, all rings are commutative noetherian of characteristic $p>0$ unless otherwise said so. We summarize some definitions and notations used in tight closure theory. Let $R^{o}$ be the complement of all minimal primes of $R$. The $\textit{tight}$ $\textit{closure}$ $I^{*}$ of an ideal $I$ of $R$ is the set of all elements $z \in R$ such that $c z^{p^{e}} \in I^{[p^{e}]}$ for $e \gg 0$ and some $c \in R^{o}$, where $I^{[p^{e}]}$ is the ideal generated by the $q=p^{e}$-th powers of an ideal $I$. $I$ is called $\textit{tightly}$ $\textit{closed}$ if $I^{*}=I$. Assume $R$ is a domain. Define $R^{+}$ to be the integral closure of $R$ in an algebraic closure of the field of fractions of $R$. Then K. Smith~\cite{Sm} proves that $I^{*}=I R^{+} \cap R$ for every parameter ideal $I$ of an excellent local domain $R$. However, in a recent breakthrough, \cite{BM}, it is proven that this equality does not hold for all ideals.

Next, assume that $R$ is reduced. Let $R^{1/q}$ be the ring obtained by adjoining all $q$-th roots of elements of $R$. Denote by $R^{\infty}$ the directed union of the tower:
$$
R \hookrightarrow R^{1/p} \hookrightarrow \cdots \hookrightarrow R^{1/p^{e}} \hookrightarrow \cdots.
$$
The $\textit{Frobenius}$ $\textit{closure}$ of an ideal $I$ of $R$ is defined to be $I^{F}:=I R^{\infty} \cap R$. $R$ is called $F$-$\textit{pure}$ (resp. $F$-$\textit{injective}$) if every ideal $I$ of $R$ (resp. for all parameter ideals $I$ of $R$; an ideal generated by $\het(I)$ number of generators) satisfies $I^{F}=I$.

Recall that a noetherian ring $R$ is $F$-$\textit{finite}$ if the Frobenius morphism $R\to R,r \mapsto r^{p} $ is finite.  By Kunz~\cite[2.5]{Ku}, every $F$-finite ring is excellent. Therefore, the formal fiber rings of an $F$-finite ring are geometrically regular.

Let $R$ be a reduced ring and let $q=p^{e}$. Note that there is an isomorphism $R^{1/q} \simeq {^eR}$ via the Frobenius map, where $^eR$ carries its usual right $R$-module structure, while its left $R$-module is via the $e$-fold iterates of the Frobenius map.

\section{Localization problem for F-pure rings}

In this section, we solve a ``slightly generalized" Grothendieck's localization problem for ``$\mathcal{P}$=maximal Cohen-Macaulay" (Proposition \ref{localization-mcm}) and then establish Theorem \ref{localization}, one of our main theorems in this paper.

To establish these results mentioned above, let us start with the definition of Radu-Andr$\mathrm{\grave{e}}$ morphism and state some known results related to it. In particular, it can be used to prove certain properties along the fibers of flat maps of noetherian rings, where cohomological arguments do not work.

\begin{Definition}
Let $\varphi:R \to S$ be a homomorphism of noetherian rings of characteristic $p>0$. Then the $Radu$-$Andr\grave{e}$ $morphism$ $w^{e}_{S/R}:S \otimes_{R} {^eR} \to {^eS}$ is the ring homomorphism defined by
$$
w^{e}_{S/R}(s \otimes r)=s^{p^{e}}\varphi(r).
$$
The commutative ring $W^{e}_{S/R}:=S \otimes_{R} {^eR}$ is called the $e$-$th$ $Radu$-$Andr\grave{e}$ $ring$. When $e=1$, we simply write $w_{S/R}$.
\end{Definition}

In general, the Radu-Andr$\mathrm{\grave{e}}$ ring is not noetherian (see~\cite{E} for other results concerning this ring). As we are mostly concerned with $F$-finite rings, we may harmlessly assume that it is noetherian, and the map $w^{e}_{S/R}$ is a finite map for any $e > 0$. Notice that any ring homomorphism $R \to S$ is identified with the natural factorization:
$$
\begin{CD}
{^eR} @>>> S \otimes_{R} {^eR} @> w^{e}_{S/R} >> {^eS}. \\
\end{CD}
$$
It is also straightforward that the composite map:
$$
\begin{CD}
S @>>> S \otimes_{R} {^eR} @> w^{e}_{S/R} >> {^eS} \\
\end{CD}
$$
is the Frobenius morphism on $S$.

Recall that a flat homomorphism of noetherian rings is $\textit{regular}$ if all fibers are geometrically regular. In connection with this map, the following theorem, originally proved by Radu \cite{Ra} and Andr$\mathrm{\grave{e}}$ \cite{An} and summarized as \cite[Theorem 2]{Du}, is the most remarkable.

\begin{Theorem}
[\cite{Du}, Theorem 2]
Let $\varphi:R \to S$ be a homomorphism of noetherian rings of characteristic $p>0$. Then it is regular if and only if the map $w_{S/R}$ is flat.
\end{Theorem}

A ring homomorphism $\varphi:R \to S$ is said to be $reduced$ (resp. $normal$) if $\varphi$ is flat and the fiber ring of $\varphi$ at every $\fp \in \Spec R$ is geometrically reduced (resp. geometrically normal) over $k(\fp)$.

\begin{Theorem}
[\cite{Du}, Theorem 3]
\label{Dumi}
Let $\varphi:R \to S$ be a homomorphism of noetherian rings. Then the following are equivalent
\begin{enumerate}
\item[$\mathrm{(1)}$]
$\varphi$ is reduced;

\item[$\mathrm{(2)}$]
$\varphi$ is flat, $w_{S/R}$ is injective, and $R[S^{p}]$ is a pure $R$-submodule of $S$;

\item[$\mathrm{(3)}$]
$w_{S/R}$ is injective and $S/R[S^{p}]$ is $R$-flat.
\end{enumerate}
\end{Theorem}

In the proof of the main theorem, we shall use the following localization theorems, both of which hold in arbitrary characteristic.

\begin{Theorem}[\cite{Ni}, 2.4]
\label{Ni}
Let $\varphi:(R,\fm,k_{R}) \to (S,\fn,k_{S})$ be a local map of noetherian rings such that
\begin{enumerate}
\item[$\mathrm{(1)}$]
the formal fibers of $R$ are reduced $($resp. normal$)$,

\item[$\mathrm{(2)}$]
$S/\fm S$ is geometrically reduced $($resp. geometrically normal$)$ over $k_{R}$, and

\item[$\mathrm{(3)}$]
$\varphi$ is flat.
\end{enumerate}
Then the map $\varphi$ is reduced $($resp. normal$)$.

\end{Theorem}

\begin{Theorem}[\cite{AF1}, Theorem 4.1]
\label{af}
Let $\varphi:(R,\fm) \to (S,\fn)$ be a flat local map of noetherian rings. Assume that the formal fibers of $R$ and the closed fiber of $\varphi$ have one of the following properties:

\begin{enumerate}
\item[$\mathrm{(CI)}$]
complete intersection.

\item[$\mathrm{(G)}$]
Gorenstein.

\item[$\mathrm{(CM)}$]
Cohen-Macaulay.
\end{enumerate}
Then all the fibers of $\varphi$ and the formal fibers of $S$ have the corresponding property.
\end{Theorem}

\begin{Definition}
Let $R$ be a noetherian ring, and let $M$ be a finite $R$-module. Then $M$ is $\textit{maximal}$ $\textit{Cohen}$-$\textit{Macaulay}$ (``MCM'' for short) if $\depth M_{\fp}=\dim R_{\fp}$ for every $\fp \in \Supp_{R} M$.
\end{Definition}

We also need the following proposition, whose proof is obtained via Theorem 3.5. Before giving the proof, we note the following simple fact: Let $R$ be a noetherian local ring. Then for non-zero finite $R$-modules $M$ and $N$, we have
$$
\depth (M \oplus N)=\min\{\depth M,\depth N\}.
$$

\begin{Proposition}
\label{localization-mcm}
Let $\varphi:(R,\fm) \to (S,\fn)$ be a flat local map of noetherian rings, and let $N$ be an $R$-flat finite $S$-module. Assume that:

\begin{enumerate}
\item[$\mathrm{(1)}$]
the formal fibers of $R$ are Cohen-Macaulay,

\item[$\mathrm{(2)}$]
the closed fiber ring $S/\fm S$ of $\varphi$ is Cohen-Macaulay, and

\item[$\mathrm{(3)}$]
$N/\fm N$ is MCM over $S/\fm S$.

\end{enumerate}
Then $N\otimes_R\kappa(\fp)$ is MCM over $S\otimes_R\kappa(\fp)$ for all $\fp\in \Spec R$. In other words, $N_{\fq}/\fp N_{\fq}$ is MCM over $S_{\fq}/\fp S_{\fq}$ for every $\fq \in \Supp_S N$, $\fp \in \Spec R$ with $\fp=R \cap \fq$.
\end{Proposition}

\begin{proof}
To prove this, we will use Nagata's trivialization to reduce the problem to the local flat map of noetherian rings and then apply Theorem \ref{af}.

Let $S*N=\{(a,m)~|~a \in S,~m \in N\}$ whose ring structure is given by $(a,m)*(b,n):=(ab,an+bm)$. As $S$-modules, $S*N$ and $S \oplus N$ are isomorphic. The subset $0*N$ of the ring $S*N$ is an ideal, which is isomorphic to $N$, and we get the short exact sequence of $S$-modules:
$$
\begin{CD}
0 @>>> 0*N @>>> S*N @>>> S @>>> 0,
\end{CD}
$$
where $S$ is naturally identified with the subring $S*0$ of $S*N$. Since $S \to S*N$ is a trivial extension as modules, the above short exact sequence splits. Furthermore, $S*N$ is a module-finite $S$-algebra.

For $\fp \in \Spec R$ and $\fq \in \Spec S$ with $\fp=R \cap \fq$, we have a (split) short exact sequence of $S_{\fq}/\fp S_{\fq}$-modules:
$$
\begin{CD}
0 @>>> (0*N)_{\fq}/\fp(0*N)_{\fq} @>>> (S*N)_{\fq}/\fp(S*N)_{\fq} @>>> S_{\fq}/\fp S_{\fq} @>>> 0.
\end{CD}
$$
In particular, if $\fq \in \Supp N$, $N_{\fq}/\fp N_{\fq} \ne 0$ and by the remark quoted above, we have
$$
\depth(N_{\fq}/\fp N_{\fq})=\depth\big((0*N)_{\fq}/\fp(0*N)_{\fq}\big)=\depth\big((S*N)_{\fq}/\fp(S*N)_{\fq}\big),
$$
and also
$$
\dim(S_{\fq}/\fp S_{\fq})=\dim\big((S*N)_{\fq}/\fp(S*N)_{\fq}\big).
$$
Since $N$ is $R$-flat by assumption, $\varphi_{N}:R \to S*N$ is a flat local map of local rings. As to the closed fiber ring $(S*N)/\fm(S*N)=(S/\fm S)*(N/\fm N)$, we have
\begin{align} 
\depth\big((S*N)/\fm(S*N)\big)&=\depth\big((S/\fm S)*(N/\fm N)\big)\notag\\
&=\min\{\depth(S/\fm S),\depth(N/\fm N)\}\notag\\
&=\dim(S/\fm S)\tag{\dag}\\
&=\dim\big((S/\fm S)*(N/\fm N)\big)=\dim\big((S*N)/\fm(S*N)\big)\tag{\ddag},
\end{align}
where (\dag) is true since $N/\fm N$ is MCM over $S/\fm S$ by assumption, and (\ddag) is true since $S/\fm S \to (S/\fm S)*(N/\fm N)$ is module-finite. Consequently, $(S*N)/\fm(S*N)=(S/\fm S)*(N/\fm N)$ is MCM over $S/\fm S$. Now we may apply Theorem \ref{af} to $\varphi_N:R \to S*N$ to conclude that $(S*N)_{\fq}/\fp(S*N)_{\fq}$ is Cohen-Macaulay, but then
\begin{align}
\depth(N_{\fq}/\fp N_{\fq})&=\depth\big((S*N)_{\fq}/\fp(S*N)_{\fq}\big)\notag\\
&=\dim\big((S*N)_{\fq}/\fp(S*N)_{\fq}\big)\notag\\
&=\dim(S_{\fq}/\fp S_{\fq})\notag,
\end{align}
which finishes the proof.
\end{proof}

The following characterization of $F$-purity for $F$-finite rings is needed to prove our main theorem (Theorem \ref{localization}).

\begin{Remark}[\cite{HR}, Corollary 5.3]
\label{Remark1}
If $R$ is an $F$-finite noetherian ring, then one verifies that the following conditions are equivalent:
\begin{enumerate}
\item $R$ is $F$-pure.
\item $R_{\fp}$ is $F$-pure for any $\fp \in \Spec R$.
\item The Frobenius morphism $R \to {^1R}$ is a pure extension.
\item The Frobenius morphism $R \to {^1R}$ splits.
\end{enumerate}
\end{Remark}

Before proving the main theorem, we need the following lemma.

\begin{Lemma}
\label{pure}
Let $(R,\fm)$ be an $F$-finite algebra over a field $K$ of characteristic $p>0$. If $K \to L$ is a finite separable extension, then $R$ is $F$-pure if and only if $R \otimes_{K} L$ is $F$-pure.
\end{Lemma}

\begin{proof}
If $R \otimes_{K} L$ is $F$-pure, it follows by (\cite{Fedd}, Proposition 1.3 (6)) that $R$ is also $F$-pure since $R$ is a direct summand of $R \otimes_{K} L$. Let $K \to L$ be a finite separable extension. Then since the extension $R \to R \otimes_{K} L$ is finite $\mathrm{\acute{e}}$tale, the Radu-Andr$\mathrm{\grave{e}}$ morphism:
$$
(R \otimes_{K} L) \otimes_{R} {^1R} \to {^1(R \otimes_{K} L)}
$$
is an isomorphism. (This isomorphism is proved in~\cite{GR} for weakly $\mathrm{\acute{e}}$tale maps.) By hypothesis on the ring $R$, the Frobenius morphism:
$$
R \otimes_{K} L \to (R \otimes_{K} L) \otimes_{R} {^1R} \to {^1(R \otimes_{K} L)}
$$
splits, which is the desired claim.
\end{proof}

Let $k(\fp)$ denote the residue field of $R$ at $\fp \in \Spec R$.

\begin{Theorem}
\label{localization}
Let $\varphi:(R,\fm,k_{R}) \to (S,\fn,k_{S})$ be a flat local map of $F$-finite rings. Assume that the closed fiber of $\varphi$ is Gorenstein and geometrically $F$-pure over $k_{R}$. Then the fiber of $\varphi$ at every $\fp \in \Spec R$ is geometrically $F$-pure over $k(\fp)$.
\end{Theorem}

\begin{proof}
In what follows, we fix the notation. Choose $\fp \in \Spec R$ and $\fq \in \Spec S$ such that $\fp=R \cap \fq$. First, observe that the closed fiber of $\varphi$ is geometrically reduced over $k_{R}$ by hypothesis, so that Theorem \ref{Ni} and Theorem \ref{af}(G) imply that the fiber of $\varphi$ at every $\fp \in \Spec R$ is Gorenstein and geometrically reduced over $k(\fp)$.

Let $A$ be a noetherian algebra over a field $K$ of characteristic $p>0$, and let $K \to L$ be a finite field extension. Then $K \to L$ splits into a sequence $K \to K' \to L$, where $K \to K'$ is purely inseparable, while $K' \to L$ is separable. By Lemma \ref{pure}, (\cite{Fedd}, Proposition 1.3 (6)), and together with the sequence of pure extensions:
$$
A \to A \otimes_{K} K' \to (A \otimes_{K} K') \otimes_{K'} L \simeq A \otimes_{K} L,
$$
it suffices to show that the base change of any fiber ring of $\varphi:R \to S$ with respect to a purely inseparable extension is $F$-pure.

Let $k(\fp) \to L$ be a finite purely inseparable extension. Then we can write $L=k(\fp)(\theta_1,\dots,\theta_t)$ for some $\theta_1,\dots,\theta_t\in L$. We claim that it suffices to show that $S\otimes_{R} k(\fp) \otimes_{k(\fp)} {^ek(\fp)}$ is $F$-pure for some integer $e > 0$, and we reason as follows. We let $p^{e_{i}}=[k(\fp)(\theta_{i}):k(\fp)]$ and $e=\max \{e_{i}~|~ 1 \le i \le t \}$. The finite extension $k(\fp)\to {^ek(\fp)}$ can be identified with $k(\fp) \to k(\fp)^{1/p^e}$. Hence $k(\fp) \to L$ can be naturally identified as a subextension of $k(\fp) \to k(\fp)^{1/p^e}$. Once $S\otimes_{R} k(\fp)\otimes_{k(\fp)} {^ek(\fp)}$ is $F$-pure, so is $S\otimes_{R} k(\fp)\otimes_{k(\fp)} L$ since $S\otimes_{R} k(\fp)\otimes_{k(\fp)} L$ is a direct summand of $S\otimes_{R} k(\fp)\otimes_{k(\fp)} {^ek(\fp)}$. Therefore, it suffices to show that $S\otimes_{R} k(\fp)\otimes_{k(\fp)} {^ek(\fp)}$ is $F$-pure. To do this, we set
$$
W^{e}_{\fq}:=S_{\fq}/\fp S_{\fq} \otimes_{k(\fp)} {^ek(\fp)} \simeq S_{\fq} \otimes_{R_{\fp}} {^ek(\fp)}.
$$
This is the Radu-Andr\'{e} ring corresponding to $k(\fp)\to S_{\fq}/\fp S_{\fq}$. Then $W^{e}_{\fq}$ is the localization of $S \otimes_{R} {^ek(\fp)}$ at $\fq \in \Spec S$. Since $S \otimes_{R} {^ek(\fp)}$ is reduced, $W^{e}_{\fq}$ is also reduced. Note, in general, that if $A$ is a local algebra over a field $k$ of characteristic $p > 0$, then $A \otimes_{k} {^ek}$ is also local. The above observations also imply the following. $S_{\fq} \to S_{\fq} \otimes_{R} {^eR}$ is a module-finite extension of local rings for any $\fq \in \Spec S$. Suppose that $W^{e}_{\fq}$ is $F$-pure for all $\fq \in \Spec S$. Since the Frobenius morphism commutes with localization, this implies that $S \otimes_{R} {^ek(\fp)}$ is $F$-pure. Hence we are reduced to showing that the local ring $W^{e}_{\fq}$ is $F$-pure.

Since $S/\fm S$ is geometrically $F$-pure over $k_{R}$ by hypothesis, the $e$-th Frobenius on $W^e_{\fn}$ splits as well. Consider the following maps of module-finite $W^{e}_{\fn}$-algebras:
$$
\begin{CD}
W^{e}_{\fn}  @>\alpha >> ^e(S/\fm S) @>\beta >> {^e(W^{e}_{\fn})},
\end{CD}
$$
where $\alpha$ is the Radu-Andr\'{e} morphism and $\beta$ is induced by the natural inclusion $S/\fm S\to W^e_{\fn}=S/\fm S\otimes_{k_R}{^ek_R}$. Then we find that
$$
\beta\alpha(\bar{s}\otimes c)=\beta(\bar{s}^{p^e}c)=\bar{s}^{p^e}c\otimes 1=\bar{s}^{p^e}\otimes c\cdot 1=\bar{s}^{p^e}\otimes c^{p^e}
$$
for $\bar{s}\in S/\fm S$ and $c\in {^ek_R}$. Therefore, $\beta\alpha$ is exactly the $e$-th Frobenius on $W^e_{\fn}$.

We claim that $\alpha$ splits and we reason as follows. Since the $e$-th Frobenius on $W^e_{\fn}$ splits, there is a map $\gamma:{^e(W^e_{\fn})}\to W^e_{\fn}$ such that $\gamma\beta\alpha=\gamma F^e_{W^e_{\fn}}=\mathrm{id}_{W^e_{\fn}}$. Consequently, $\alpha$ splits. (Actually, one can show that, if $\alpha$ splits, so does $F^e_{W^e_{\fn}}$. Assume that there is a map $\alpha':{^e(S/\fm S)}\to W^{e}_{\fn}$ such that $\alpha'\alpha=\mathrm{id}_{W^e_{\fn}}$. Note that $\beta$ splits since it is clearly a finitely generated free extension. Let $\beta':{^e(W^{e}_{\fn})}\to {^e(S/\fm S)}$ be such that $\beta'\beta=\mathrm{id}_{^e(S/\fm S)}$. Then it is evident that $\alpha'\beta'F^e_{W^e_{\fn}}=\alpha'\beta'\beta\alpha=\mathrm{id}_{W^e_{\fn}}$. Consequently, the $e$-th Frobenius on $W^e_{\fn}$ splits.)

Let $N:=\Coker(w^{e}_{S/R})$. Then by Theorem \ref{Dumi}, the sequence of $W^{e}_{S/R}$-modules:
\begin{equation}
\label{radu-andre-sequence}
0 \to S \otimes_{R} {^eR}  \xrightarrow{w^{e}_{S/R}} {^eS} \to N \to 0
\end{equation}
is exact, and $N$ is an $R$-flat finite $W^{e}_{S/R}$-module. Furthermore, the ring $W^{e}_{\fq}$ is Gorenstein for every $\fq \in \Spec S$ because $S_{\fq}/\fp S_{\fq}$ is. In particular, we have $K_{W^{e}_{\fq}} \simeq W^{e}_{\fq}$, which is the canonical module of $W^{e}_{\fq}$.

We claim that ${^e(S/\fm S)}$ is MCM over $W^{e}_{\fn}$ and we reason as follows. Let $d:=\dim(S/\fm S)$. Notice that $W^e_{\fn}$ is module-finite over $S/\fm S$, so we have $d=\dim(W^e_{\fn})$. Since $S/\fm S$ is Gorenstein, every system of parameters $\bar{s}_1,\dots,\bar{s}_d$ of $S/\fm S$ is regular. But then $\bar{s}^{p^e}_1,\dots,\bar{s}^{p^e}_d$ also form a regular sequence on $S/\fm S$. Since the image of $\bar{s}_{k}\otimes 1$ under the map $W^{e}_{\fn} \to {^e(S/\fm S)}$ is just $\bar{s}^{p^e}_k$ for $1 \le k \le d$, it follows that $\bar{s}_1\otimes 1,\dots,\bar{s}_d\otimes 1$ form a regular sequence on ${^e(S/\fm S)}$, which is the required claim.

Let $\alpha:W^{e}_{\fn} \to {^e(S/\fm S)}$ be the same map as above. Then since
$$
W^e_{\fn} \simeq S \otimes_{R} k_{R} \otimes_{k_R} {^ek_R} \simeq  S \otimes_R{^ek_R} \simeq S\otimes_R {^eR}\otimes_{^eR} {^ek_{R}},
$$
we find that $\alpha=w^{e}_{S/R} \otimes \mathrm{id}_{^ek_{R}}$ by applying $(-)\otimes_{^eR} {^ek_{R}}$ to the map $w^{e}_{S/R}:S\otimes_R {^eR}\to {^eS}$. Thus, applying $(-)\otimes_{^eR} {^ek_{R}}$ to the exact sequence (\ref{radu-andre-sequence}), we have an exact sequence of $W^{e}_{\fn}$-modules:
$$
\begin{CD}
0 @>>> W^{e}_{\fn} @>\alpha >> {^e(S/\fm S)} @>>> N/\fm N @>>> 0,
\end{CD}
$$
which splits, as we have already observed. Then
$$
\depth_{W^e_{\fn}}\big({^e(S/\fm S)}\big)= \min\{\depth(W^e_{\fn})=\dim(W^e_{\fn}),\depth_{W^e_{\fn}}(N/\fm N)\}.
$$
But since $\depth_{W^e_{\fn}}\big({^e(S/\fm S)}\big)=\dim(W^e_{\fn})$ and $\depth_{W^e_{\fn}}(N/\fm N)\leq \dim(W^e_{\fn})$, we have
$$
\depth_{W^e_{\fn}}(N/\fm N) \le \dim(W^e_{\fn})=\depth_{W^e_{\fn}}\big({^e(S/\fm S)}\big) \le \depth_{W^e_{\fn}}(N/\fm N),
$$
which implies that $N/\fm N$ is MCM over $W^{e}_{\fn}$. Proposition \ref{localization-mcm} applied to the flat local map $^eR \to S \otimes_{R} {^eR}$ implies that $N_{\fq}/\fp N_{\fq}$ is also MCM over $W^{e}_{\fq}$ for every $\fq \in \Supp_{S} N$. We also have
$$
\ext^{1}_{W^{e}_{\fq}}(N_{\fq}/\fp N_{\fq},K_{W^{e}_{\fq}}) \simeq \ext^{1}_{W^{e}_{\fq}}(N_{\fq}/\fp N_{\fq}, W^{e}_{\fq})=0,
$$
due to the local duality (\cite{BH}, Theorem 3.3.10). On the other hand, the sequence of $W^{e}_{\fq}$-modules:
$$
\begin{CD}
0 @>>> W^{e}_{\fq} @>>> ^e(S_{\fq}/\fp S_{\fq}) @>>> N_{\fq}/\fp N_{\fq} @>>> 0
\end{CD}
$$
is obtained by applying $(-)\otimes_{^eR} {^ek(\fp)}$ to the sequence (\ref{radu-andre-sequence}) and then localizing at $\fq$. This is a short exact sequence. Since $\ext^{1}_{W^{e}_{\fq}}(N_{\fq}/\fp N_{\fq}, W^{e}_{\fq})=0$, it is straightforward that $W^{e}_{\fq}\to {^e(S_{\fq}/\fp S_{\fq})}$ splits, which implies that the $e$-th Frobenius on $W^{e}_{\fq}$ splits since the $e$-th Frobenius on $W^{e}_{\fq}$ is the same as (the proof is exactly the same as when $\fp=\fm$)
$$
\begin{CD}
W^{e}_{\fq} @>>> {^e(S_{\fq}/\fp S_{\fq})} @>>> {^e(W^{e}_{\fq})}. \\
\end{CD}
$$
This proves that $W^{e}_{\fq}$ is $F$-pure and hence completes the proof of the theorem.
\end{proof}

\begin{Remark}
\begin{enumerate}
\item
Theorem \ref{localization} holds for $F$-injective rings in place of $F$-pure rings, since these notions are equivalent to each other for Gorenstein local rings (\cite{Fedd}, Lemma 3.3). It is easy to give an example of an $F$-pure ring, which is not geometrically $F$-pure. Let $K$ be a field of characteristic $p>0$ and let $L=K^{1/p}$. Then $L$ is not geometrically reduced over $K$ as follows. Let $a \in K$ be such that $K'=K[t]/(t^{p}-a)$ is a field. However, we get $L \otimes_{K} K' \simeq L[t]/(t-b)^{p}$ with $b^{p}=a$. We also mention that Enescu~\cite{E1} constructed an algebra that is $F$-injective, but is not geometrically $F$-injective.
\item
Note that the localization problem may be formulated for a ring homomorphism of (not necessarily local) noetherian rings in the following way. If all the closed fibers for a flat ring homomorphism have a property $\mathcal{P}$, then do all the fibers of the map have also $\mathcal{P}$? The localization problem can be generalized to this set-up if, for example, a (global) property $\mathcal{P}$ for a ring is defined on its local rings. In particular, Proposition \ref{localization-mcm} and Theorem \ref{localization} hold for flat ring homomorphisms of noetherian rings without any modifications in the proof.
\end{enumerate}
\end{Remark}

\section{Geometric Consequences}

In this section, we establish some geometric consequences of results that we have proved in the previous section.

Let $\varphi:R \to S$ be a ring homomorphism. For a property $\mathcal{P}$, we define $U_{\varphi}(\mathcal{P})$ to be the set of all $\fp \in \Spec R$ such that the fiber $S \otimes_{R} k(\fp)$ has $\mathcal{P}$. Under certain hypotheses, it is known that the set $U_{\varphi}(\mathcal{P})$ can possess some nature with respect to the Zariski topology for many interesting cases of $\mathcal{P}$ (some related results are found in~\cite{FOV}).

We recall preliminary notations and results that we shall use in the following. A subset of a noetherian scheme $X$ is $\textit{constructible}$ if it is written as a disjoint union of finitely many locally closed subsets of $X$. A subset $U \subset X$ is open if and only if $U$ is constructible and is stable under generization. We record the following:

\begin{Lemma}
\label{flat}
Let $(R,\fm,k_{R}) \to (S,\fn,k_{S})$ be a local map of noetherian rings, and let $N$ be a finite $S$-module. Then the following hold:

\begin{enumerate}
\item[$\mathrm{(1)}$] If both $S$ and $N$ are $R$-flat, then $\depth_{S} N=\depth_{R} R +\depth_{S}(N/\fm N)$.

\item[$\mathrm{(2)}$]
Assume that $k_{R} \to k_{S}$ is a finitely generated extension of fields. Then $N \otimes_{R} k_{R}$ is MCM over $S \otimes_{R} k_{R}$ if and only if $N \otimes_{R} k_{S}$ is MCM over $S \otimes_{R} k_{S}$.
\end{enumerate}
\end{Lemma}

\begin{proof}
For the first, this is a special case of (\cite{BH}, Proposition 1.2.16).

For the second, it suffices to note that $S \otimes_{R} k_{S} \simeq (S \otimes_{R} k_{R}) \otimes_{k_{R}} k_{S}$ is faithfully flat over $S \otimes_{R} k_{R}$, and $S \otimes_{R} k_{S}$ is noetherian by assumption.
\end{proof}

\begin{Proposition}[Generic principle I]
\label{generic-mcm}
Let $\varphi:R \to S$ be a flat map of finite type of excellent rings. Let $N$ be an $R$-flat finite $S$-module. Assume $R$ admits a dualizing complex. Then the set of all $\fp \in \Spec R$ such that $N \otimes_{R} k(\fp)$ is MCM over $S\otimes_{R} k(\fp)$ forms an open subset of $\Spec R$.
\end{Proposition}

\begin{proof}
Suppose that $N \otimes_{R} k(\fp)$ is MCM over $S \otimes_{R} k(\fp)$ for $\fp \in \Spec R$. Then it suffices to find an open subset $\fp \in U \subseteq \Spec R$ such that $N \otimes_{R} k(\fp')$ is MCM over $S \otimes_{R} k(\fp')$ for every $\fp' \in U$.

First, assume that $R$ is CM. Choose any $\fq \in \Spec S$ with $\fp=R \cap \fq$. By Lemma 4.1 and the hypothesis for $R$, we have
$$
\depth_{S_{\fq}}N_{\fq}=\depth_{S_{\fq}}(N_{\fq}/\fp N_{\fq})+\dim R_{\fp}.
$$
Hence $N_{\fq}/\fp N_{\fq}$ is MCM over $S_{\fq}/\fp S_{\fq}$ $\iff$ $N_{\fq}$ is MCM over $S_{\fq}$. Let $S * N$ be the trivial extension. Then the natural map $S \to S*N$ gives a one-one correspondence of prime ideals of respective rings; that is, the image of $\fq \in \Spec S$ under the map $S \to S*N$ is contained in a unique prime ideal $\fq*N$. To ease the notation, we will write $\fq$ for $\fq*N$. It follows that $(S*N)_{\fq}=S_{\fq}*N_{\fq}$,
$$
\depth_{S_{\fq}}N_{\fq}=\depth_{S_{\fq}}(S_{\fq}*N_{\fq})=\depth_{S_{\fq}*N_{\fq}}(S_{\fq}*N_{\fq}),
$$
and thus $N_{\fq}/\fp N_{\fq}$ is MCM over $S_{\fq}/\fp S_{\fq}$ $\iff$ $S_{\fq}*N_{\fq}$ is CM.

Let $U $ be the maximal CM locus of $\Spec(S*N)$. $U$ is a non-empty open subset since $S$ is excellent. Let $W$ be the complement of $U$ in $\Spec(S*N)$ and let $\varphi_{N}:R \to S \to S*N$ be the composite map. Note that the map $\varphi_{N}$ is flat and of finite type. Denote by $\varphi^{*}_{N}$ the associated scheme map. But then the flatness of $\varphi_{N}$ implies that the complement $W'$ of $\varphi^{*}_{N}(W)$ has the property that $W'$ is the maximal subset of $\Spec R$ such that the fiber of $\varphi^{*}_{N}:\Spec(S*N) \to \Spec R$ at every point of $W'$ is CM. Since $\fp \in W'$, it is non-empty. By Chevalley's theorem, $W'$ is constructible. On the other hand, Proposition \ref{localization-mcm} implies that $W'$ is stable under generization. Hence $W'$ is an open set and $\fp \in W'$. Until now, we did not use the assumption that $R$ has a dualizing complex.

Next, assume that $R$ is an arbitrary excellent ring. We may reduce to the CM case as follows. Let $f:X \to \Spec R$ be a Macaulayfication (\cite{Ka1}, Theorem 1.1). Note that $f$ is a dominant and proper map, hence it is surjective. Then we have a base change diagram:
$$
\begin{CD}
X \times_{R} S @>>> X \\
@VVV @VV f V \\
\Spec S @>\varphi^{*}>> \Spec R \\
\end{CD}
$$
(By abuse of notations, $X \times_{R} S$ denotes the fiber product of schemes). Note that since $X$ is of finite type over $R$, we find that $X \times_{R} S \to X$ is an affine flat map of finite type of excellent schemes. Since the issue is local on $X$, we may assume that $X=\Spec R'$, and thus $X \times_{R} S$ is also affine. Let $f(x)=\fp$ and let $M$ denote the pull-back of $N$ under the map $X \times_{R} S \to \Spec S$. Then $M=R' \otimes_{R} N$, which is an $R'$-flat finite $R' \otimes_{R} S$-module.

Then Lemma \ref{flat} implies that $N \otimes_{R} k(\fp)$ is MCM over $S \otimes_{R} k(\fp)$ $\iff$ $M \otimes_{R'} k(x) \simeq N \otimes_{R} k(x)$ is MCM over $(R' \otimes_{R} S) \otimes_{R'} k(x) \simeq S \otimes_{R} k(x)$. Hence we may argue as previously to find a maximal open subset $U \subseteq X$ such that $M \otimes_{R'} k(x)$ is MCM over $(R' \otimes_{R} S) \otimes_{R'} k(x)$ for every $x \in U$. Then Chevalley's theorem together with Proposition \ref{localization-mcm} imply that $f(U)$ is open, which is the desired open set. This completes the proof.
\end{proof}

\begin{Remark}
By his celebrated theorem of Macaulayfication, Kawasaki was able to establish a conjecture of Sharp (\cite{Ka2}, Corollary 1.4) that states that a noetherian ring $R$ of positive dimension admits a dualizing complex if and only if $R$ is a homomorphic image of a finite-dimensional Gorenstein ring. So the above proposition applies to any affine domains over a perfect field, or their localizations.
\end{Remark}

Finally, we are able to prove the following theorem via Proposition 4.2.

\begin{Theorem}[Generic principle II]
\label{generic-f-pure}
Let $\varphi:R \to S$ be a flat map of finite type of $F$-finite rings. Assume that:

\begin{enumerate}
\item[$\mathrm{(1)}$]
all the fibers of $\varphi$ are Gorenstein,

\item[$\mathrm{(2)}$]
$R$ admits a dualizing complex, and

\item[$\mathrm{(3)}$]
$\mathcal{P}$=geometrically $F$-pure.
\end{enumerate}
Then $U_{\varphi}(\mathcal{P})$ forms a Zariski open subset of $\Spec R$.
\end{Theorem}

\begin{proof}
We fix $\fp \in \Spec R$. A key ingredient for the proof already appeared in Theorem \ref{localization}. Let us employ the notations as in Theorem \ref{localization} and let $N=\Coker(w^{e}_{S/R})$. Then for $\fq \in \Spec S$ with $\fp=R \cap \fq$, we find that $N_{\fq}/\fp N_{\fq}$ is MCM over $W^{e}_{\fq}$ $\iff$ The Radu-Andr$\mathrm{\grave{e}}$ ring $W^{e}_{\fq}$ is $F$-pure. Assume that this holds for all $\fq$ with $\fp=R \cap \fq$. Then we have that $N \otimes_{R} k(\fp)$ is MCM over $S \otimes_{R} k(\fp)$, or equivalently, the fiber of $\varphi$ over $\fp$ is $F$-pure. The theorem then follows from Proposition \ref{generic-mcm}.
\end{proof}


\begin{thebibliography}{99}

     \bibitem{An}
        M. Andr$\mathrm{\grave{e}}$,  \emph{Localisation de la
        lissit$\acute{e}$ formelle},  Manuscripta Math. \textbf{13}
        (1974),  297--307.

     \bibitem{AF1}
        L. Avramov and H.-B. Foxby,  \emph{Grothendieck's localization
        problem},  Contemp. Math. \textbf{159} (1994),  1--13.

     \bibitem{AF2}
        L. Avramov and H.-B. Foxby,  \emph{Cohen-Macaulay properties of ring
        homomorphisms},  Adv. Math. \textbf{133} (1998), 54--95.

     \bibitem{BI}
       A. Brezuleanu and C. Ionescu, \emph{On the localization theorems
       and completions of $\mathrm{\mathbf{P}}$-rings},
       Rev. Roumania Math. Pures Appl. \textbf{29}  (1984), 371--380.

     \bibitem{BH}
        W. Bruns and J. Herzog,  \emph{Cohen-Macaulay rings},
        Cambridge University Press. \textbf{39}  (1998).

     \bibitem{BM}
        H. Brenner and P. Monsky, \emph{Tight closure does not
     commute with localization}, arXiv:0710.2913.

     \bibitem{Du}
       T.  Dumitrescu,  \emph{Reducedness, formal smoothness,
       and approximation in characteristic $p$},
       Comm. Algebra  \textbf{23}  (1995), 1787--1795.

     \bibitem{E}
       F. Enescu,  \emph{On the behavior of $F$-rational rings under
       flat base change}, J. Algebra  \textbf{233}  (2000),
       543--566.

     \bibitem{E1}
       F. Enescu, \emph{Local cohomology and F-stablity}, preprint.

     \bibitem{Fedd}
       R. Fedder,  \emph{$F$-purity and rational singularity},
       Trans. Amer. Math. Soc. \textbf{278}, (1983),  461--480.

     \bibitem{FOV}
       H. Flenner, L. O'Carroll, and W. Vogel,  \emph{Joins and Intersections},
       Springer-Verlag, (1999).

     \bibitem{GR}
        O. Gabber and L. Ramero,  \emph{Almost ring theory},
        LMN  \textbf{1800}, Springer-Verlag (2003).

     \bibitem{Gro}
        A. Grothendieck,  \emph{$\acute{E}$l$\acute{e}$ments de
        g$\acute{e}$om$\acute{e}$trie alg$\acute{e}$brique} IV, Publ. Math.
        IHES \textbf{24}  (1964).

     \bibitem{HR}
        M. Hochster and J. Roberts, \emph{The purity of the Frobenius and local         cohomology}, Adv. Math. \textbf{21}, 117-172 (1972).

     \bibitem{Hu}
        C. Huneke,  \emph{Tight closure and its applications}, CBMS Lecture
        Notes in Mathematics,  Vol.\textbf{88}, Amer. Math. Soc.,
        Providence, (1996).

     \bibitem{Ka1}
        T. Kawasaki, \emph{On Macaulayfication of Noetherian schemes},
        Trans. Amer. Math. Soc. \textbf{352} (2000), 2517--2552.

     \bibitem{Ka2}
        T. Kawasaki, \emph{On arithmetic Macaulayfication of
        Noetherian rings},
        Trans. Amer. Math. Soc. \textbf{354} (2002), 123--149.

     \bibitem{Ku}
        E. Kunz, \emph{On Noetherian rings of characteristic $p$},
        Amer. J. Math. \textbf{98} (1976), 999--1013.

     \bibitem{Ni}
        J. Nishimura, \emph{On ideal adic completions of noetherian rings},
        J. Math. Kyoto  Univ. (1981), 153--169.

     \bibitem{Ra}
        N. Radu, \emph{Regularity, flatness and approximation property},
        An. Univ. Bucvre\c sti  \textbf{40} (1991), 77--81.

     \bibitem{Sm}
        K. Smith, \emph{Tight closure of parameter ideals},
        Invent. Math. \textbf{115}  (1994), 41--60.

\end{thebibliography}
\end{document}